\documentclass[11pt]{article}
\usepackage{amsmath,amsfonts,amscd,amssymb}
\usepackage{url}

\newcommand{\R}{\ensuremath{\mathbb{R}}}
\newcommand{\pp}{\ensuremath{\mathbb{P}}}

\newcommand{\so}{\ensuremath{\mathfrak{so}}}
\newcommand{\g}{\ensuremath{\mathfrak{g}}}
\newcommand{\h}{\ensuremath{\mathfrak{h}}}

\DeclareMathOperator{\tr}{tr}
\DeclareMathOperator{\grad}{grad}

\DeclareMathOperator{\spa}{span}
\DeclareMathOperator{\dd}{d}

\DeclareMathOperator{\stab}{stab}

\DeclareMathOperator{\Ad}{Ad}

\DeclareMathOperator{\id}{id}

\newtheorem{theorem}{Theorem}
\newtheorem{proposition}[theorem]{Proposition}
\newtheorem{remark}[theorem]{Remark}
\newtheorem{example}[theorem]{Example}
\newtheorem{corollary}[theorem]{Corollary}
\newtheorem{definition}[theorem]{Definition}
\newtheorem{lemma}[theorem]{Lemma}
\newenvironment{proof}{{\bf Proof:} }{\mbox{ }\hfill$\square$}
\def\etal{\textit{et al.\ }}
\def\frameA{\mbox{$\{A\}$}}
\def\frameB{\mbox{$\{B\}$}}

\begin{document}
\title{Observer design for invariant systems with homogeneous observations}
\author{C.~Lageman%
\thanks{C. Lageman is with the Department of Electrical Engineering and Computer Science, 
        Universit{\'e} de Li{\`e}ge,
        B-4000 Li{\`e}ge Sart-Tilman,
        Belgium,
        {\tt\footnotesize christian.lageman@montefiore.ulg.ac.be}. His work was supported by
        the Australian Research Council Centre of Excellence for
        Mathematics and Statistics of Complex Systems.},
        J.~Trumpf%         
\thanks{J. Trumpf is with the Department of Information Engineering,
        The Australian National University, Canberra ACT 0200,
        Australia,
        {\tt\footnotesize Jochen.Trumpf@anu.edu.au}.}
        and R.~Mahony%
\thanks{R. Mahony is with the Department of Engineering,
        The Australian National University, Canberra ACT 0200,
        Australia,
	    {\tt\footnotesize Robert.Mahony@anu.edu.au}.}}
\maketitle

\begin{abstract}
This paper considers the design of nonlinear observers for
invariant systems posed on finite-dimensional connected Lie groups
with measurements generated by a transitive group action on an associated
homogeneous space.  We consider the case where the group action
has the opposite invariance to the system invariance and show that
the group kinematics project to a minimal realisation of the
systems observable dynamics on the homogeneous output space.  The
observer design problem is approached by designing an observer for
the projected output dynamics and then lifting to the Lie-group. A
structural decomposition theorem for observers of the projected
system is provided along with characterisation of the invariance
properties of the associated observer error dynamics.  We propose
an observer design based on a gradient-like construction that
leads to strong (almost) global convergence properties of
canonical error dynamics on the homogeneous output space.  The
observer dynamics are lifted to the group in a natural manner and
the resulting gradient-like error dynamics of the observer on the
Lie-group converge almost globally to the unobservable subgroup of
the system, the stabiliser of the group action.
\end{abstract}

%%%%%%%%%%%%%%%%%%%%%%%%%%%%%%%%%%%%%%%%%%%%%%%%%%%%%%%%%%%%%%%%
\section{Introduction}\label{sec:1_Intro}

There is a growing body of work concerned with the design of
nonlinear observers for invariant systems on Lie groups. Recent
results are partly motivated by the need for highly robust and
computationally simple state estimation algorithms for robotic
vehicles.  The classical approach to state estimation for such
applications is based on nonlinear filtering techniques such as
extended Kalman filters \cite{2005_Anderson} or particle filters
\cite{2001_Doucet}. Non-linear observers offer less information
than a nonlinear filter (state-estimates rather than full
\emph{posterior} distributions for the state), however, it is
often possible to prove strong stability results with large
(almost global) basins of attraction and provide computationally
simple implementations of the observers. Historically, work on
nonlinear observers for invariant systems was divided into a body
of applied work, closely linked to specific examples, and a
separate body of theoretical work concerned with fundamental
systems theory for invariant systems. One of the earliest applied
results concerned the design of a nonlinear observer for attitude
estimation of a rigid-body using the unit quaternion
representation of the special-orthogonal Lie group
\cite{Sal1991_TAC}.  This work is seminal to a series of papers
undertaken over the last fifteen years that develop nonlinear
attitude observers for rigid-body dynamics
\cite{1999_Nijmeijer,ThiSan2003_TAC,Maithripala2004,Metni2005a,2005_MahHamPfl-C64,Bonnabel2006_acc,Metni_etal2006_CEP,2008_MahHamPfl.TAC,2007_Tayebi_cdc,2007_Martin_cdc,2008_Vasconcelos.ifac},
exploiting either the unit quaternion group structure or the
rotation matrix Lie-group structure of $SO(3)$.  Recent observer
designs have comparable performance to state-of-the-art nonlinear
filtering techniques \cite{2007_Crassidis.JGCD}, generally have
much stronger global stability and robustness properties
\cite{2008_MahHamPfl.TAC}, and are simple to implement.  The full
pose estimation problem has also attracted recent attention
\cite{VikFos_2001_CDC,RehGho2003_TAC,2007_BalHamMahTru_ECC,2007_Vasconcelos_cdc}
in which case the underlying state space is the special Euclidean
group $SE(3)$ comprising both attitude and translation of a
rigid-body.  Another promising body of applied work involves
development of heading reference systems for UAV systems
\cite{2007_Martin_cdc,2008_Martin.ifac}.  The theoretical
investigation of observability and controllability of invariant
systems on Lie-groups was first considered in the seventies
\cite{b:72,js:72,jk:81} in the context of the formulation of
systems theory on Lie-groups and coset spaces. Parallel work by
Sussmann investigated quotient structure of realisations of
nonlinear systems and associated observability and controllability
properties \cite{sushi:74,sushi:77}.  There have been only
occasional extensions of this material since the early work, for
example \cite{cdm:90,ak:96}.  A characteristic of all the work
mentioned above is that authors consider only matched invariance
assumptions on the system. That is, all the invariances in the
system are of the same handedness; left invariant kinematics, left
invariant outputs, left invariant metrics, etc. Recent work on
understanding the generic structure of observers for left
invariant systems on Lie-groups
\cite{Bonnabel2005_ch,2006_HamMah_ICRA,2008_MahHamPfl.TAC,2008_Bonnabel.TAC}
has challenged this formulation and lead to an understanding of
the importance of invariance properties of observer error dynamics
and the underlying properties of the observer design problem
\cite{2008_Lageman_mtns,2008_Bonnabel.ifac,toappear_Lageman.TAC}.
The present paper contains further results in this direction.

In this paper, we consider the design of nonlinear observers for
state space systems where the state evolves on a finite
dimensional, connected Lie group and the measurements are
generated by a transitive group action acting on an associated
homogeneous space.  We consider \emph{complementary} output
invariance; that is, for a left invariant system we consider a
right invariant output group action, a structure that is suggested
by a body of recent work
\cite{2008_Bonnabel.TAC,2008_MahHamPfl.TAC,2008_Bonnabel.ifac,toappear_Lageman.TAC}.
We show that a system with complementary outputs projects to a
quotient system on the homogeneous output space corresponding to
the minimal observable realisation of the full system, and
consequently, that the stabiliser of the output group action is
the unobservable subgroup of the system.  Since the projected
system is a minimal realisation of the full system it is natural
to study the observer design problem for this system.  We
introduce the concept of synchrony of observer and plant
\cite{toappear_Lageman.TAC} for projected systems and use this to
define a ``template'' for observer design as a combination of an
internal model and an innovation term.  The invariance properties
of the innovation term that ensure the error dynamics of the
observer/system couple are autonomous are also fully
characterised.  The second part of the paper introduces a
constructive process for observer design based on the design of
gradient-like innovation terms.  We consider cost functions
defined on the homogeneous output space and design a gradient-like
innovation for the projected system. This design can be lifted to
an observer on the Lie-group state space of the full system. Under
suitable assumptions on the cost function we obtain (almost)
global asymptotic stability of the observer on the homogeneous
output space and corresponding asymptotic stability to the
unobservable subgroup on the Lie-group state-space.  An example on
$SO(3)$ is provided to demonstrate the applications of results
provided in the body of the paper.

After the introduction, Section~\ref{sec:2_problem_formulation}
provides an overview of the systems considered and notation used.
Section~\ref{sec:3_structure} studies the structure of invariant
systems with complementary invariance in the outputs.
Section~\ref{sec:4_induced_props} studies the properties of
observers for the projected system on the homogeneous output
space. Section~\ref{sec:5_design} considers the question of
observer synthesis.  First we consider synthesis of observers for
the projected system (Subsection \ref{sec:5_design_induced}) and
then the lift of these observers onto the Lie-group state-space
(Subsection~\ref{sec:5_design_full}).  Section \ref{sec:6_example}
provides an example based on the attitude observers that have been
studied recently
\cite{Bonnabel2005_ch,2005_MahHamPfl-C64,2006_HamMah_ICRA,2008_MahHamPfl.TAC,2008_Bonnabel.TAC}.
A short paragraph of conclusions is also provided in
Section~\ref{sec:7_conclusions}.

%%%%%%%%%%%%%%%%%%%%%%%%%%%%%%%%%%%%%%%%%%%%%%%%%%%%%%%%%%%%%
\section{Invariant systems with homogeneous output space}\label{sec:2_problem_formulation}

In this section, we introduce the notation used throughout the
paper and discuss some of the intuition in the models considered.

The symbol $G$ denotes a connected Lie group with Lie algebra
$\g$.   We consider left invariant systems on $G$ of the form
\begin{equation}\label{eq:sys_state}
\dot X = X u.
\end{equation}
An input $u \colon \R \rightarrow \g$ is called admissible if the
resulting time-variant differential equation on the Lie group has
unique solutions for all initial values.  We assume that the
velocity $u$ is measured and can be used in the observer
construction.  The system \eqref{eq:sys_state} is termed left
invariant since it is invariant under application of left
multiplication on the group $T_X L_Y \dot{X} = (Y X) u $ where
$L_Y : X \mapsto Y X$ and $T_X$ denotes the tangent map evaluated
at $X$. An equivalent definition is used for right
invariance with respect to right multiplication $R_Y : X \mapsto X
Y$ on the group.

A smooth, connected manifold is denoted $M$ and a Lie-group action on $M$ is
denoted $h\colon G\times M\rightarrow M$, where the $h$ notation
anticipates that the group action will be acting as an output map.
A group action on a manifold is called a right action if for all
$X, Y\in G$, $y\in M$: $h(X,h(Y,y))= h(YX,y)$ and a left action if
$h(X,h(Y,y))=h(XY,y)$.  All actions considered in this paper are
transitive, that is, for any $x, y \in M$ there exists $X \in G$
such that $h(X,x) = y$.

We choose a fixed a point $y_0\in M$ to be the `reference' output.
We consider outputs for the system \eqref{eq:sys_state} generated
by the group action $y = h(X,y_0)$, $y, y_0 \in M$. In particular,
for the fixed reference $y_0$ we define $h_{y_0} : G\rightarrow M$,
$y_0 \in M$, and the output
\begin{equation}\label{eq:sys_out}
y = h_{y_0}(X) = h(X,y_0).
\end{equation}
We use the terminology \textbf{homogeneous output space} to refer
to a homogeneous manifold $M$ associated with the scenario
described above. 
In general
the group kinematics considered may be either left or right
invariant and the group action can be independently either a left
or right action.  We term the case where the outputs are generated
by a group action of opposite invariance to the system invariance
as \textbf{complementary observations}.  Conversely, we term the
case where both the system and outputs have the same invariance as
\textbf{matched observations}.  That is, right invariant outputs
are complementary to left invariant group kinematics and matched
with right invariant group kinematics, and \textit{vice versa} for
left invariant outputs.

The stabiliser of an element $y \in G$ is given by
\[
\stab(y) = \{ X\in G\;|\; h(X,y)=y\},
\]
and is a subgroup of $G$. We will mostly be interested in
$\stab(y_0)$ and 
assume that it is a closed subgroup of $G$.
We use the notation $h_X: M \rightarrow M$ for the symmetry map
\[
h_X (y) := h(X,y).
\]

While we have chosen to work explicitly with left invariant system
dynamics on the group \eqref{eq:sys_state}, the case of right
invariant dynamics is analogous and all the results in the paper
will carry through with the obvious changes.  The choice of
complementary or matched output invariance, however, is a key
assumption.  The main results of the paper require that the output
observations are complementary and the results do not hold for
matched observations. It is instructive to consider an example to
provide some physical intuition for this choice.

\begin{example} \label{ex:attitude_1}
We use the attidue estimation problem as an illustration \cite{Bonnabel2006_acc,2008_Bonnabel.TAC,2008_Bonnabel.ifac,Bonnabel2005_ch,2007_Crassidis.JGCD,2006_HamMah_ICRA,2008_Lageman_mtns,toappear_Lageman.TAC,2005_MahHamPfl-C64,2008_MahHamPfl.TAC,Maithripala2004,2004_Maithripala_acc,2008_Martin.ifac,2007_Martin_cdc,Metni_etal2006_CEP,Metni2005a,Sal1991_TAC,2007_Tayebi_cdc,TayMcG_2006_CST,ThiSan2003_TAC,2008_Vasconcelos.ifac,VikFos_2001_CDC}.
The attitude of a rigid body, with body-fixed-frame $\frameB$,
measured with respect to an inertial frame $\frameA$, can be
identified with an element $R$ of $SO(3)$.  The left invariant
dynamics
\begin{equation}\label{eq:ex:sys}
\dot{R} = R \Omega_\times
\end{equation}
on $SO(3)$ correspond to the natural body-fixed-frame kinematics
of the system.  Here $\Omega = (\Omega_1, \Omega_2, \Omega_3)^\top
\in \frameB$ is the angular velocity of the rigid-body expressed
in the body-fixed-frame and
\[
\Omega_\times = \left(\begin{array}{ccc}
0 & -\Omega_3 & \Omega_2 \\
\Omega_3 & 0 & -\Omega_1 \\
-\Omega_2 & \Omega_1 & 0
\end{array} \right).
\]

Consider the case where only partial information on the rigid-body
attitude is measured.  The most common situation is when an
inertial direction, such as the magnetic field, or gravitational
field, is measured by sensor systems such as magnetometers or
accelerometers on board a vehicle.  The inertial direction $y_0
\in \frameA$ of the measured direction is known \emph{a-priori},
for example, the gravitational direction lies in the $z$-axis of
the inertial frame.  The measured direction in the
body-fixed-frame,
\[
y = R^\top y_0 \in \frameB,
\]
is obtained by using the attitude matrix to transform the inertial
direction $y_0$ from the inertial frame into the body fixed frame.
The associated output group action that we consider is
\[
h : G \times S^2 \rightarrow S^2, \quad h(R, y_0) = R^\top y_0,
\]
a right-invariant group action.  This example demonstrates the
case of complementary outputs, typical for applications involving
robotic vehicles where measurements are made by embarked sensor
systems.

An alternatively scenario involves a fixed sensor system that
observes an autonomous vehicle in motion.  An example would be a
computer vision system that extracts a principle axis of symmetry
of an observed vehicle from a video sequence of images.  In this case the
known data is a direction $\bar{y}_0 \in \frameB$ that denotes the
orientation of the symmetry of the vehicle in the
body-fixed-frame.  The measurement is obtained by transforming
$y_0$ into the inertial frame
\[
\bar{y} = R y_0 \in \frameA.
\]
The associated group action would be
\[
\bar{h} : G \times S^2 \rightarrow S^2, \quad h(R, \bar{y}_0) = R \bar{y}_0,
\]
a left-invariant group action.  This second example demonstrates
the case of matched observations.  There are fewer applications where
matched observations are the natural model.
\end{example}

The complementary case of a left invariant system
\eqref{eq:sys_state} with right output action \eqref{eq:sys_out}
is the case that has attracted interest recently
\cite{2008_Bonnabel.TAC,2008_Lageman_mtns,2008_Bonnabel.ifac,toappear_Lageman.TAC}
and is most important for applications in robotic vehicles.  It
turns out, perhaps somewhat counter-intuitively, that the
structure of invariant systems with complementary outputs is
relatively straightforward compared to the case of matched
outputs.  A number of recent works have shown that the case of
complementary outputs yields autonomous error dynamics when the
observer is properly designed
\cite{2008_Lageman_mtns,2008_Bonnabel.ifac,2008_Bonnabel.TAC,toappear_Lageman.TAC}
and this work has lead to some promising applied results
\cite{2007_Martin_cdc,2008_Martin.ifac}. Our work can be viewed as
providing further perspective and extension of this research
direction.  The case of matched observations was studied in some
detail in the seventies \cite{b:72,js:72,jk:81} although observer
design was not a focus of this early work.

\begin{remark}\label{rem:control}
In the system \eqref{eq:sys_state} and \eqref{eq:sys_out}, the map
$u \mapsto X u$ is surjective and the system is fully
controllable. In fact, the controllability of the system is not of
particular interest in the main body of the paper except in the
context of the realisation results we present.  The results will
hold equally well for systems with restricted inputs
(corresponding to limited degrees of control actuation) as long as
system is controllable.   
If the system is not controllable then
the results presented in the paper apply to the system restricted
to the maximal controllable subgroup.  In this case the
homogeneous output space considered must be restricted to the
orbit of the maximal controllable subgroup.
\end{remark}

%%%%%%%%%%%%%%%%%%%%%%%%%%%%%%%%%%%%%%%%%%%%%%%%%%%%%%%%%%%%%%%%%%%%%%%%%%%%%%%%%%%%%%%%%%%%
\section{Structure of left invariant systems with complementary observations}\label{sec:3_structure}

In this section, we consider the structure of left invariant
systems with complementary observations.  The main result shows
that there is a natural projection of the full system dynamics
onto a system on the homogeneous output space that is a minimal
realisation of the input-output behaviour of the full system, and
that this in turn characterises the unobservable subspace on $G$
as $\stab(y_0)$.

In the remainder of this paper, we consider left-invariant system
kinematics \eqref{eq:sys_state} and complementary right invariant
`output' group action $h$ \eqref{eq:sys_out}.  Any system on a
manifold $N$ given globally by $\dot{x}= f(x,u)$ can be considered
as a bundle map $f\colon B\rightarrow TN$ from a trivial smooth
fiber bundle $B$ over $N$ to the tangent bundle $TN$.  
In general any bundle map from a fiber bundle $B$ over $N$
to $TN$ can be considered as a realization of a control system.
We will use
the notion of a system with symmetry \cite{s:81,gm:85,ns:85}. A
system $f\colon B\rightarrow TN$ with output map $h\colon
N\rightarrow M$ has a symmetry $H$ (where $H$ is a Lie group) if
there are left actions $S^B\colon H\times B \rightarrow B$,
$S^N\colon H\times N\rightarrow N$ and a right action $S^M\colon
H\times M \rightarrow M$, such that for all $X\in H$, $v\in B$,
$x\in M$
\begin{eqnarray*}
f(S^B(X,v)) &=& T S_X^N f(v), \\
h(S^N(X,x)) &=& S^M(X,h(x)) .
 \end{eqnarray*}
Here $S^N_X$ denotes the map $v\mapsto S^N(X,v)$.

\begin{remark}
Note that the systems with symmetries considered in
\cite{gm:85,ns:85} do not have specific output maps. The work
\cite{s:81} considers symmetries for systems with outputs,
however, the groups considered were commutative and the difference
between left and right symmetries was lost.  The definition we use
is extended in the natural manner to conform to the structure we
consider.  The factorisation results for system with symmetries
\cite{gm:85,ns:85,s:81}, as we will use here, have also motivated
recent work on general quotient systems \cite{tp:05}.
\end{remark}

\begin{proposition}
The stabiliser $\stab(y_0)$ is a symmetry of system
\eqref{eq:sys_state} with complementary observations
\eqref{eq:sys_out}.
\end{proposition}

\begin{proof}
The system \eqref{eq:sys_state} and \eqref{eq:sys_out} can be
regarded as a bundle map $f\colon G\times \g\rightarrow TG$ via
$f(X,u)= Xu$.  Define the left actions $S^B(Y,(X,u))=(YX,u)$ and
$S^N(Y,Xu)=YX u$ of $\stab(y_0)$ on $G\times \g$ and $TG$. Note
additionally that $\stab(y_0)$ acts on $M$ via $S^M(Y,y) = h(Y,y)
= y$. The symmetry condition on $f$ follows trivially from the
definition, while the symmetry condition on $h$ follows from the
definition of the stabiliser.
\end{proof}

The main structural property of the system \eqref{eq:sys_state}
with complementary observations \eqref{eq:sys_out} is that it
\emph{projects to a system on $M$}; that is, the system on $G$
maps via the group action to a system on $M$ with trivial output
map and this projected system is strongly equivalent to
\eqref{eq:sys_state} and \eqref{eq:sys_out} in the sense of
\cite{sushi:77}. In other words, the system on $G$ can be
projected to a system on $M$ with the same input-output behaviour.

\begin{theorem} \label{thm:lr}
The system \eqref{eq:sys_state} with complementary observations
\eqref{eq:sys_out} projects to the system
\begin{equation}\label{eq:induced}
\dot{y} = T_X h_{y_0} (Xu), \quad X\in h_{y_0}^{-1}(y)
\end{equation}
on $M$.
\end{theorem}

\begin{proof}
Since the control system $f\colon G\times \g \rightarrow TG$,
$f(X,u)=Xu$ has the symmetry $\stab(y_0)$ which acts freely and
properly, it can be projected to a control system $\tilde{f}$ on
the corresponding orbit spaces \cite{ns:85}. As $\stab(y_0)$ acts
on $G\times\g$ and $TG$ by the corresponding left actions, we can
identify the orbit spaces $\stab(y_0)\diagdown G\times\g$ and
$\stab(y_0)\diagdown TG$ with $M\times \g$ and $TM$ by setting
$(\stab(y_0) X,u)= (h(X,y_0),u)$ and $\stab(y_0)(Xu) = T_X
h_{y_0}(Xu)$. The projected control system on $M$ can be written
as $\tilde{f}\colon M\times\g \rightarrow TM$, $\tilde{f}(y,u) =
T_X h_{y_0}(Xu)$ for $h(X,y)=y_0$. Note that this shows that
\eqref{eq:induced} is well-defined and smooth. The operations
above can be illustrated by the following commutative diagram:
$$
\begin{CD}
G\times \g @>>> \stab(y_0)\diagdown G\times\g @= M\times \g \\
@V{Xu}VV  @VVV @V{T_Xh_{y_0}(Xu))}VV \\
TG @>>> \stab(y_0)\diagdown TG @= TM
\end{CD}
$$
Since the output map $h$ is equivariant under the action of
$\stab(y_0)$ it induces a map $\tilde{h}$ between the orbit spaces
$\stab(y_0)\diagdown G$ and $M\diagup\stab(y_0)$. As $\stab(y_0)$
acts only trivially on $M$, we have that $M\diagup\stab(y_0) = M$.
The quotient $\stab(y_0)\diagdown G$ can be identified with $M$
via $\stab(y_0)X = h_{y_0}(X)$. The induced map $\tilde{h}$ has to
satisfy the commutative diagram
$$
\begin{CD}
G @>h_{y_0}>> M \\
@V{h_{y_0}}VV @V{\id}VV\\
M @>\tilde{h}>> M
\end{CD}
$$
where $\id_M$ denotes the identity.
Thus $\tilde{h}=\id_M$.
In summary we have the following commutative diagram
$$
\begin{CD}
G\times\g @>{Xu}>> TG @>{\pi_G}>> G @>{h_{y_0}}>> M\\
@VVV @VVV @VVV @V{\id_M}VV \\
M\times\g @>{T_Xh_{y_0}(Xu)}>> TM @>{\pi_M}>> M @>{\id_M}>> M\\
\end{CD}
$$
where $\pi_G\colon TG\rightarrow G$, 
$\pi_M\colon TM\rightarrow M$ are the canonical
projections of the tangent bundle.
It follows from the uniqueness of solutions for
admissible inputs that the system \eqref{eq:sys_state} and
\eqref{eq:sys_out} and the system \eqref{eq:induced} with full
state output have the same input-output behaviour.
\end{proof}

Theorem \ref{thm:lr} provides a significant insight into the
observability of the full system.  In particular, since the system
can be reduced to a fully-state observable system on the
observation space, the observed dynamics cannot provide any
additional information on the system state.

Recall that two states $X, Y\in G$ are called indistinguishable,
if for any admissible input $u$ the solutions $X(t)$, $Y(t)$ of
\eqref{eq:sys_state} and \eqref{eq:sys_out}
 produce the same output
$h(X(t),y_0) = h(Y(t),y_0)$ \cite{b:72,sushi:77}.

\begin{corollary}
Consider the system \eqref{eq:sys_state} with complementary
observations \eqref{eq:sys_out}.  Two states $X,Y\in G$ are
undistinguishable if and only if $X Y^{-1}$ is contained in the
stabiliser subgroup $\stab(y_0)$ of $y_0$. In particular, a state
$X\in G$ is undistinguishable from the identity in $G$ if and only
if $X\in \stab(y_0)$.
\end{corollary}

\begin{proof} Since the projected system on $M$ is fully observable,
two states in $X,Y\in G$ are indistinguishable
if and only if they are mapped onto the same $y\in M$.
Two states $X,Y\in G$ are mapped onto the same $y\in M$
if and only if $X Y^{-1}\in \stab(y_0)$.
\end{proof}

\begin{remark}
One can view Theorem \ref{thm:lr} as a special case of Sussmann's
result \cite{sushi:77} of the existence of a minimal realisation.
Sussmann's proof is based on factoring the state space by the
equivalence relation of indistinguishable states. In our case, two
states $X_1$, $X_2$ are indistinguishable if and only if $X_1
X_2^{-1}\in\stab(y_0)$. Hence, factoring by this equivalence
relation is equivalent to factoring by $\stab(y_0)$. This is
another way to see that projecting the system on $G$ onto $M$
yields a minimal realisation of \eqref{eq:sys_state} and
\eqref{eq:sys_out}.
\end{remark}

%%%%%%%%%%%%%%%%%%%%%%%%%%%%%%%%%%%%%%%%%%%%%%%%%%%%%%%%%%%%%%%%%%%%%
\section{Observers for projected systems on homogeneous spaces}\label{sec:4_induced_props}

%In the previous section we have seen that the observable part of
%the system is given by the projected system on the homogeneous
%space. This induced system on the homogeneous space has a very
%simple structure: one observes the full state and the velocity of
%the trajectory. The problem of constructing an observer for a
%system on a Lie group with full state and velocity measurements
%was addressed in \cite{toappear_Lageman.TAC}. This suggests to
%extend the construction used there to the projected system on
%homogeneous spaces and lift the observer on $M$ to the group $G$
%to obtain an observer for \eqref{eq:sys_state} and
%\eqref{eq:sys_out}

In this section we consider the structure of nonlinear observers
for the projected system on the homogeneous output space.  We show
that the concepts of synchrony and canonical errors, and a
decomposition of observers into internal model and innovation
terms discussed in earlier work by the authors
\cite{2008_Lageman_mtns,toappear_Lageman.TAC} can be extended to
the case of complementary observations.   We provide a full
characterisation of the invariance properties of the innovation
term that guarantees invariance of the error dynamics, an
extension of recent work by Bonnabel \etal
\cite{2008_Bonnabel.ifac}.

Consider systems of the form
\begin{equation}\label{eq:sys:hom}
\dot{y} = T_X h_{y_0}( Xu) \text{ with } X\in h_{y_0}^{-1}(y)=\{ X\in G\;|\; h_{y_0}(X)=y\}
\end{equation}
on the homogeneous space $M$. We call such systems
\textbf{projected systems} since they arise from a projection of
the system $\dot{X}=Xu$ on $G$ as discussed in Theorem
\ref{thm:lr}.  A projected system, however, is defined on the
homogeneous space $M$ and can be independently analysed.  A simple
example is instructive.

\begin{example}\label{ex:attitude_2}
Consider Example~\ref{ex:attitude_1} of attitude estimation on
$SO(3)$ for the case of a complementary observation on the
homogeneous space $S^2$.  The measured output is modelled by $y =
R^\top y_0$ for a known `reference' $y_0 \in S^2$.  Consider an
arbitrary $R_y \in SO(3)$ such that $R_y^\top y_0 = y$ and note
that $R_y^\top y_0 = y = R^\top y_0$.  For any such $R_y$ the
projected system dynamics are given by
\[
\dot{y} = T_{R_y} h_{y_0} \left(R_y \Omega_\times \right) =
-\Omega_\times R_y^{\top} y_0  = -\Omega_\times y,
\]
%on the homogeneous output space $S^2$.  Thus, the projected system is
%\[
%\dot{y} = -\Omega_\times y
%\]
the kinematics of an inertial direction on $S^2$ expressed with
respect to the body-fixed-frame.  The projected system can be
studied independently as a kinematic system $\dot{y} =
-\Omega_\times y$ on $S^2$ without reference to the Lie-group
kinematics $\dot{R} = R \Omega_\times$ \cite{Metni_etal2006_CEP,Metni2005a}.
\end{example}

The first concept we consider for observer structure for projected
systems is that of synchrony of a plant and an observer. We use
the notation $e \colon M \times M\rightarrow N$, to denote a
smooth {\bf error function} function to a smooth manifold $N$. An
error function is a generalisation of the observer error $(\hat{x}
- x)$ that is used for systems on $\R^n$. For the moment we make
no further assumptions on the error function.

\begin{definition}
Let $M$ be a smooth manifold, $B$ a vector bundle over $M$,
$f_x\colon B\rightarrow TM$, $f_{\hat{x}}\colon B\rightarrow TM$
bundle maps and
\begin{equation}\label{eq:sync:pair}
\begin{split}
\dot{x} =& f_x(u) \\
\dot{\hat{x}} =& f_{\hat{x}}(u)
\end{split}
\end{equation}
a pair of general systems on $M$.  We call the pair of systems
\eqref{eq:sync:pair} {\bf e-synchronous} with respect to an error
function $e \colon M \times M\rightarrow N$ if for all admissible
$u\colon \R\rightarrow\R^n$, all initial values
$x_0$, $\hat{x}_0\in M$, and all $t\in \R^+$
\[
\frac{d}{dt} e(\hat{x}(t;\hat{x}_0,u),x(t;x_0,u)) =0 .
\]
\end{definition}

\begin{theorem}\label{thm:hom:err}
Consider a pair of projected systems on $M$,
\begin{align}\label{eq:sync}
\begin{split}
\dot{x} &= T_X h_{y_0}\left( Xu\right), \quad X\in h_{y_0}^{-1}(x) \\
\dot{\hat{x}} &= T_{\hat{X}} h_{y_0}\left( \hat{X}u  \right) , \quad \hat{X}\in h_{y_0}^{-1}(\hat{x}) .
\end{split}
\end{align}
If there exists a smooth error function $e\colon M\times
M\rightarrow N$ such that the systems are $e$-synchronous, then
$e$ has the form
\[
e(\hat{x},x) = g( h_{y_0}( \hat{X}X^{-1} )) \text{ with } h_{y_0}(\hat{X})=\hat{x}, h_{y_0}(X)=x
\]
and $g\colon M\rightarrow N$ a smooth function.
\end{theorem}

\begin{proof}
We first show that $e$ has to be invariant under the action of $G$
on $M\times M$, i.e. $e(h(S,\hat{x}),h(S,x))=e(\hat{x},x)$ for all
$x,\hat{x}\in M$, $S\in G$. Let $S\in G$ and $x_0,\hat{x}_0\in M$.
We choose a smooth curve $T\colon [0,1]\rightarrow G$, $T(0)=e$,
$T(1)=S$ and set $u(t)= T(t)^{-1} T(t)$.  Note that for any $x\in
M$, $t\in [0,1]$, $R\in G$ with $h(R,y_0)=x$ we have
\begin{eqnarray*}
\frac{d}{dt} h(T(t),x)  &=& \frac{d}{dt} h(T(t),h(R,y_0)) \\
&=& \frac{d}{dt} h_{y_0}(R T(t)) \\
&=& T_{R T(t)} h_{y_0} (R \dot{T}(t)) \\
&=& T_{R T(t)} h_{y_0} (R T(t) u(t)) .
\end{eqnarray*}
Hence $h(T(t),x)$ is a solution of $\dot{x} =T_X h_{y_0}( Xu),
X\in h_{y_0}^{-1}(x)$ with initial value $x$ and $u(t)=T(t)^{-1}
T(t)$. Thus $(h(T(t),x_0),h(T(t),\hat{x}_0))$ is the solution of
\eqref{eq:sync} with inital value $(x_0,\hat{x}_0)$ and $u(t)$ as
above. Since the systems are $e$-synchronous, $e$ is constant on
solutions of \eqref{eq:sync} and $e(x_0,\hat{x}_0) =
e(h(T(1),x_0),h(T(1),\hat{x}_0))$. Thus for all $S\in G$
\[
e = e\circ h_S
\]
and $e$ is invariant under the action of $G$. For $x,\hat{x}\in M$
let now $X,\hat{X}\in G$ such that
$h_{y_0}(X)=x,h_{y_0}(\hat{X})=\hat{x}$. Then
\[
e( \hat{x},x ) = e(h(\hat{X},y_0),h(X,y_0)) = e(h(X^{-1},h(\hat{X},y_0)),y_0)=e(h_{y_0}(\hat{X}X^{-1}),y_0).
\]
Thus $e$ has indeed the form as claimed in the theorem with a
smooth $g\colon M\rightarrow N$, $g(z)= e(z,y_0)$.
\end{proof}

Theorem~\ref{thm:err:hom} justifies the consideration of a {\bf
canonical error function}
\begin{equation}\label{eq:errfun}
e(\hat{x},x) = h(\hat{X}X^{-1},y_0) \text{ with } h_{y_0}(\hat{X})=\hat{x}, h_{y_0}(X)=x.
\end{equation}
The canonical error function is well-defined, smooth and
non-degenerate, in the sense that the differential of $e$ with
respect to either the first or second variables is full rank.  Note
that ``no error'' corresponds to $e(\hat{x},x)=y_0$.

\begin{remark}
The canonical error is the projection of the right-invariant error
considered in the Lie group case \cite{toappear_Lageman.TAC} onto
the homogeneous space $M$.   The right-invariant error was
associated with synchrony of left-invariant systems on the Lie
group.
\end{remark}

\begin{definition}
Two systems on a homogenous space are {\bf synchronous}, if they
are $e$-synchronous with respect to the canonical error
\eqref{eq:errfun}.
\end{definition}

\begin{theorem}\label{thm:sync:hom}
Consider the projected system \eqref{eq:sys:hom} on a homogeneous space $M$,
and a general second system on $M$ of the form
\[
\dot{\hat{y}} = f_{\hat{y}}(\hat{y},u,t) .
\]
Then the pair of systems are synchronous if and only if
\[
f_{\hat{y}}(\hat{y},u,t) = T_{\hat{X}} h_{y_0}( \hat{X}u),
\]
for all $\hat{X} \in G$ such that $h(\hat{X}, y_0) = \hat{y}$.
\end{theorem}

\begin{proof}
Consider firstly a pair of systems
\begin{subequations}\label{eq:proj}
\begin{align}
\dot{y} =& T_{X} h_{y_0}( Xu) \label{eq:proj_X} \\
\dot{\hat{y}} =& T_{\hat{X}} h_{y_0}( \hat{X}u). \label{eq:proj_hatX}
\end{align}
\end{subequations}
Consider the systems on $G$
\begin{subequations}\label{eq:lift}
\begin{align}
\dot{X} = & Xu, & h(X(0),y_0) = y(0) \label{eq:lift_X}  \\
\dot{\hat{X}}  =& \hat{X}u, & h(\hat{X}(0),y_0) = \hat{y}(0). \label{eq:lift_hatX}
\end{align}
\end{subequations}
It is straightforward to see that the solutions of \eqref{eq:lift}
project to the solutions of \eqref{eq:proj}.  From the Lie group
case \cite{toappear_Lageman.TAC}, it is known that the systems
\eqref{eq:lift} are synchronous with respect to the right
invariant error $E_r = \hat{X}X^{-1}$.  Hence their projected
solutions \eqref{eq:proj} must be synchronous with respect to the
projection of $E_r$, the canonical error on $M$.  It follows that
the plants are synchronous.

On the other hand, consider a pair of \emph{synchronous} systems
\begin{eqnarray*}
\dot{y} &=& T_{X} h_{y_0}( Xu) \\
\dot{\hat{y}} &=& f_{\hat{y}}(\hat{y},u,t).
\end{eqnarray*}
Consider the system \eqref{eq:lift_X} along with
\[
\dot{\hat{X}} = \left(f_{\hat{y}}(\hat{y},u,t)\right)^H,
\]
where $v^H$ denotes\footnote{A construction of a suitable
horizontal distribution is given in the proof of
Theorem~\protect\ref{thm:inv:inno}.} the horizontal lift with
respect to an arbitrary, smooth horizontal distribution to the
fibres of $M$ in $G$.  For the right invariant error on $G$, one
has
\begin{eqnarray*}
\frac{d}{dt} \hat{X}X^{-1} &=& T_{\hat{X}} R_{X^{-1}}  \left(f_{\hat{y}}(\hat{y},u,t)\right)^H
- T_{X^{-1}} L_{\hat{X}} T_e L_{X^{-1}} T_{X} R_{X^{-1}}  X u \\
&=& T_{\hat{X}} R_{X^{-1}}  \left(f_{\hat{y}}(\hat{y},u,t)\right)^H
- T_{\hat{X}} R_{X^{-1}} T_{e} L_{\hat{X}} u .\\
\end{eqnarray*}
Since these systems project down onto the original systems on $M$,
the equivariance of $h_{y_0}$ yields that
\begin{eqnarray*}
\frac{d}{dt} h_{y_0}(\hat{X}X^{-1})
&=& T_{\hat{X}X^{-1}} h_{y_0} T_{\hat{X}} R_{X^{-1}}  \left(f_{\hat{y}}(\hat{y},u,t)\right)^H
- T_{\hat{X}X^{-1}} h_{y_0} T_{\hat{X}} R_{X^{-1}} \hat{X} u \\
&=&
 T_{\hat{y}} h_{X^{-1}} T_{\hat{X}} h_{y_0}  \left(f_{\hat{y}}(\hat{y},u,t)\right)^H
-  T_{\hat{y}} h_{X^{-1}} T_{\hat{X}} h_{y_0} \hat{X} u \\
&=& T_{\hat{y}} h_{X^{-1}} f_{\hat{y}}(\hat{y},u,t)
-  T_{\hat{y}} h_{X^{-1}} T_{\hat{X}} h_{y_0} \hat{X} u .
\end{eqnarray*}
Since this result holds for all initial values $y(0)\in M$, the
synchrony condition of the systems yield
\[
f_{\hat{y}}(\hat{y},u,t) =   T_{\hat{X}} h_{y_0}(\hat{X} u).
\]
\end{proof}

From the definition of synchrony it is possible to develop and
decomposition of an observer into a synchronous internal model and
an innovation or nonlinear output injection term.

\begin{definition}\label{def:int_model_innov}
Consider a pair of systems $f_y\colon B\times\R\rightarrow TM$,
$f_{\hat{y}}\colon B\times M\times\R\rightarrow TM$,
\begin{eqnarray*}
\dot{y} &=& f_y(u,t) \\
\dot{\hat{y}} &=& f_{\hat{y}}(w,y,t)
\end{eqnarray*}
on a homogeneous space.
\begin{enumerate}
\item%%
We say that $\hat{y}$ has an internal model of $y$ if for all
admissible inputs, all $y_0\in M$ and all $t\in\R^+$
\[
\hat{y}(t,y_0,y(t,y_0,u),u) = y(t,y_0,u)
\]

\item %%
We define an innovation term to be a map $\alpha\colon M\times TM\times M \rightarrow TM$ such that
\begin{enumerate}
\item $\alpha(\hat{y},w,y,t)\in T_{\hat{y}} M$
\item $\alpha(\hat{y}(t,y_0,y(t,y_0,u),u),f_y(u,t),y(t,y_0,u),t)=0$ for all $y_0\in M$, $t\in\R$ and all admissible inputs
$u$.
\end{enumerate}
\end{enumerate}
\end{definition}

\begin{proposition}
Consider the projected system \eqref{eq:sys:hom} on $M$.
Then any observer with an internal model of this system has the form
\begin{equation}\label{eq:gobs:hom}
\dot{\hat{y}} = T_{\hat{X}} h_{y_0}( \hat{X} u ) + \alpha(\hat{y},u,y,t)
\end{equation}
where $\alpha$ is an innovation term.
\end{proposition}

\begin{proof}
Consider an observer
$$
\dot{\hat{y}} = f(\hat{y},u,y,t)
$$
with an internal model of the observed system.
Then we can define
$$
g(\hat{y},u,y,t) = f(\hat{y},u,y,t) - T_{\hat{X}} h_{y_0}( \hat{X} u )  .
$$
A calculation analogous to the proof of Theorem \ref{thm:sync:hom} shows that
$$
\frac{d}{dt} h_{y_0}(\hat{X}X^{-1})
=  T_{\hat{X}X^{-1}} h_{y_0} T_{\hat{X}} R_{X^{-1}} g(\hat{y},u,y,t).
$$
Since $\hat{y}$ has an internal model of $y$, this implies that
$$
g(y,u,y,t)=0
$$
for all $y\in M$, $t\in\R$ and admissible $u\in \g$, i.e. $g$ is an innovation.
\end{proof}

We call an innovation term $\alpha\colon M\times\g\times M\times\R \rightarrow TM$ {\bf equivariant} if for all $S\in G$,
$y,\hat{y}\in M$, $u\in\g$, $t\in\R$
$$
T_{\hat{y}} h_S \alpha(\hat{y},u,y,t) = \alpha(h_S(\hat{y}),u,h_S(y),t).
$$
Note that in this definition we consider only the specific action
$(S,\hat{y},u)\mapsto (h_S(\hat{y}),u)$ on the control bundle
$B=M\times \g$.

\begin{theorem}\label{thm:inv:inno}
The dynamics of the canonical error of an observer with an
internal model, \eqref{eq:gobs:hom}, is autonomous if and only if
the innovation term $\alpha$ is equivariant and does not depend on
$u$ or $t$. The autonomous error dynamics has the form
\[
\frac{d}{dt} e = \alpha(e,y_0).
\]
\end{theorem}

\begin{proof}
Choose a subspace $\h\subset \g$ such that $\g = \ker T_e h_{y_0}
+ \h$ is the direct sum of $\h$ and the Lie-algebra of the
stabiliser subgroup.  Define a horizontal distribution $H$ on $G$
by $H(X):= T_e R_X \h$. Let $u\colon \R\rightarrow \g$ be a fixed,
admissible input and $y(t)$, $\hat{y}(t)$ two solutions of
\eqref{eq:sys:hom} and \eqref{eq:gobs:hom} corresponding to this
input. We choose $X_0,\hat{X}_0\in G$ such that $h_{y_0}(X_0)=y(0)$
and $h_{y_0}(\hat{X}_0)=\hat{y}(0)$. Let $X(t)$ and $\hat{X}(t)$ be
curves with initial values $X_0$, $\hat{X}_0$, which satisfy the
differential equations
\begin{eqnarray*}
\dot{X} &=&   X u  \\
\dot{\hat{X}} &=& \hat{X} u  - \left(\alpha(\hat{y},u,y,t)\right)^H, \\
\end{eqnarray*}
where $(v)^H$ denotes the horizontal lift of a vector $v\in T_{h_{y_0}(X)} M$ to $T_X G$ via $H(X)$.
The curves project to $y(t)$ and $\hat{y}(t)$, i.e. $h_{y_0}(X(t))=y(t)$ and $h_{y_0}(\hat{X}(t))=\hat{y}(t)$ hold.
Let us first consider the right-invariant error $\hat{X}X^{-1}$ on the Lie group.
Then we get that
\begin{eqnarray*}
\lefteqn{\frac{d}{dt} \hat{X}X^{-1} }\\
&=& T_{\hat{X}} R_{X^{-1}}   \hat{X} u  + T_{\hat{X}} R_{X^{-1}}\left(\alpha(\hat{y},u,y,t)\right)^H
- T_{X^{-1}} L_{\hat{X}} T_e L_{X^{-1}} T_{X} R_{X^{-1}}   X u  \\
&=&  T_{\hat{X}} R_{X^{-1}}\left(\alpha(\hat{y},u,y,t)\right)^H
\end{eqnarray*}
Using the equivariance of the canonical projection we get for the canonical error
on the group that
\begin{eqnarray*}
\frac{d}{dt} e &=& T_{\hat{X}X^{-1}} h_{y_0}\left( T_{\hat{X}} R_{X^{-1}}\left(\alpha(\hat{y},u,y,t)\right)^H\right) \\
&=& T_{\hat{y}} h_{X^{-1}} T_{\hat{X}} h_{y_0}\left(\left(\alpha f(\hat{y},u,y,t)\right)^H\right) \\
&=& T_{\hat{y}} h_{X^{-1}} \alpha(\hat{y},u,y,t) .
\end{eqnarray*}
If $\alpha$ is equivariant and does not depend on $u$ or $t$, then
\[
\frac{d}{dt}e = \alpha(h(X^{-1},\hat{y}),y_0) = \alpha(e,y_0),
\]
i.e. the dynamics of $e$ is autonomous.
On the other hand, if the dynamics of $e$ is autonomous, then there is a function $F\colon M\rightarrow TM$ with
for all $e, y,\hat{y}\in M$ $X,\hat{X}\in G$ with $h_{y_0}(X)=y$, $h_{y_0}(\hat{X})=\hat{y}$, $h_{y_0}(\hat{X}X^{-1})=e$, that
$$
F(e) = T_{\hat{y}} h_{X^{-1}} \alpha(\hat{y},u,y,t) .
$$
It follows immediately that $\alpha$ does not depend on $u$ or $t$,
i.e. $\alpha(\hat{y},u,y,t)=\alpha(\hat{y},y)$.
Since the canonical error is invariant under the natural action of $G$ on $M\times M$,
one has for all $S\in G$ that
$$
T_{h_S(\hat{y})} h_{(SX)^{-1}} \alpha(h_S(\hat{y}),h_S(y)) = F(e) = T_{\hat{y}} h_{X^{-1}} \alpha(\hat{y},y).
$$
In particular, this holds for $y=y_0$.
Thus, for all $S\in G$, $\hat{y}\in M$
\[
T_{\hat{y}} h_S \alpha(\hat{y},y_0) = \alpha(h_S(\hat{y}),h_S(y_0)) .
\]
For $x, y \in M$, $R\in G$ with $y= h(Y,y_0)$, we can set
$\hat{y}=h(Y^{-1},x)$ and $S=R Y$, and obtain
$$
T_y h_R \alpha(x,y) = \alpha(h_R(x),h_R(y)) .
$$
Hence $\alpha$ is equivariant.
\end{proof}

%%%%%%%%%%%%%%%%%%%%%%%%%%%%%%%%%%%%%%%%%%%%%%%%%%%%%%%%%%%%%%%%%%%%%%%%%%%%%%%%%%%%%
\section{Observer synthesis}\label{sec:5_design}

In this section we consider the question of observer synthesis.
The approach taken is to build observers from a synchronous term
along with an gradient-like innovation derived from a cost
function on the homogeneous space.  Thus, we begin by considering
gradient-like observers for projected systems and then look at
lifting these observers up to the Lie group.

%%~~~~~~~~~~~~~~~~~~~~~~~~~~~~~~~~~~~~~~~~~~~~~~~~~~~~~~~~~~~~~~~~~~~~~~~~~~~~~~
\subsection{Observer synthesis on the projected system}\label{sec:5_design_induced}

In order to compute gradient-like terms for construction of the
observer it is necessary to define a Riemannian metric on the
homogeneous output space.  Furthermore, it is natural, and indeed
necessary for the results that follow, that the metric is
invariant with respect to the group action. Not all homogeneous
spaces necessarily admit such a metric and it is necessary to
assume this additional structure to continue with the proposed
approach.  That is, we assume that the homogeneous output space
$M$ is a \textbf{reductive homogeneous space}, or equivalently
that it admits an invariant Riemannian metric
$\langle\cdot,\cdot\rangle$ on $M$ \cite{ce}.  One has that for
all $S \in G$, $y\in M$, $u,v\in T_y M$
\[
\langle u,v\rangle =\langle T_y h_S u, T_y h_S v\rangle.
\]

Let $f\colon M \times M \rightarrow \R$ be a smooth cost function
on $M$. That is, the diagonal $\Delta=\{(y,y)\;|\;y\in M\}$ is a
closed minimal level set of the function $f$.  A smooth cost
function $f\colon M\times M\rightarrow \R$ is said to be invariant
function with respect to the group action $h$ if
\[
f(h(S,\hat{y}),h(S,y))=f(\hat{y},y)
\]
for all $y,\hat{y}\in M$, $S\in G$.

We propose the observer design
\begin{equation}\label{eq:obs:hom}
\dot{\hat{y}} = T_{\hat{X}} h_{y_0}( \hat{X} u ) - \grad_1 f(\hat{y},y)
\end{equation}
for the system \eqref{eq:sys:hom} on $M$.  Here, $\grad_1 f$
denotes the gradient of the function $\hat{y} \mapsto
f(\hat{y},y)$.  Note that since $f$ is minimal on the diagonal
$\Delta$ then $\hat{y} \mapsto f(\hat{y},y)$ has a local minima at
$\hat{y} = y$ and $\grad_1 f(y,y) = 0$. This shows that the
gradient term is an innovation according to
Definition~\ref{def:int_model_innov}.  The following lemma
shows that the innovation $\grad_1 f$ is itself equivariant if
the cost function and metric are invariant.

\begin{lemma}\label{lem:grad}
Consider a group action $h : G \times M \rightarrow M$ on a
reductive homogeneous space $M$.  Let $g$ be an invariant metric
on $M$ and $f$ be an invariant cost function with respect to $h$.
Then for all $S\in G$, $y,\hat{y}\in M$
\[
\grad_1 f( h(S,\hat{y}), h(S,y)) = T_{\hat{y}} h_S \grad_1 f(\hat{y},y).
\]
That is $\grad_1 f$ is an equivariant vector field on $M$.
\end{lemma}

\begin{proof}
By the invariance of $f$ and the Riemannian metric we have for all
$v\in T_{\hat{y}} M$
\begin{eqnarray*}
\langle \grad_1 f(\hat{y},y) , v\rangle &=& \dd_1 f(\hat{y},y)(v) \\
&=& \dd_1 f(h(S,\hat{y}),h(S,y))(T_{\hat{y}} h_S(v)) \\
&=& \langle \grad_1 f(h(S,\hat{y}),h(S,y)),T_{\hat{y}} h_S(v)\rangle \\
&=& \left\langle \left(T_{\hat{y}} h_S\right)^{-1} \grad_1 f(h(S,\hat{y}),h(S,y)), v \right\rangle,
\end{eqnarray*}
with $\dd_1$ denoting the differential with respect to the first variable.
\end{proof}

Since any equivariant innovation term yields autonomous error
dynamics (Theorem \ref{thm:inv:inno}), we obtain the following
corollary for the proposed gradient-like observer
\eqref{eq:obs:hom}.

\begin{theorem}\label{thm:err:hom}
Consider system \eqref{eq:sys:hom} on a reductive homogeneous
space equipped with an invariant metric.  Assume that there is an
invariant cost function $f$ and consider the observer
\eqref{eq:obs:hom}.   Then the canonical error $e =
h(\hat{X}X^{-1}, y_0)$ has the dynamics
\[
\dot{e} = - \grad_1 f(e,y_0).
\]
\end{theorem}

\begin{proof}
By Lemma \ref{lem:grad} the term $\grad_1 f$ is equivariant. By
Theorem \eqref{thm:inv:inno} the error is autonomous and the the
form as given above.
\end{proof}

Since the projected system is fully observed and the observer
design is based on a gradient construction, it is possible to
obtain strong almost global stability results for the observer.

\begin{corollary}\label{cor:conv:hom}
Consider system \eqref{eq:sys:hom} on a reductive homogeneous
space equipped with an invariant metric.  Assume that there is an
invariant cost function $f$ and consider the observer
\eqref{eq:obs:hom}.  Assume furthermore that $\hat{y}\mapsto
f(\hat{y},y_0)$ is a Morse-Bott function with unique global
minimum $y_0$ and no further local minima.  Then the error $e$
converges to $y_0$ for almost all initial conditions and
arbitrary, admissible inputs $u$.
\end{corollary}

It is relevant to note that although the innovation term is
equivariant the observer itself is not.  This is due to the fact
that the synchronous term is not equivariant with respect to
direct to direct application of the group action.  This situation
is analogous to the case observed for Lie group systems with full
measurements \cite{toappear_Lageman.TAC}.

Since the system state of the projected system is fully observed
then the choice of a gradient innovation term based on a function
of the full state on $M$ is very natural.  The most significant
difficulty in the proposed observer design is the requirement to
find a suitable cost function $f$. The two requirements for a cost
function candidate are the invariance property (used in the
structure of the observer design) and the Morse-Bott requirement
(used in the convergence proof Corollary~\ref{cor:conv:hom}). The
easiest approach to generating cost functions is to start by
searching for a Morse-Bott function candidate without demanding
the invariance property.  That is one looks to construct a
function $\hat{f} : M \rightarrow \R$ such that $\hat{f}(y)$ is
Morse-Bott and has an isolated global minima at $y_0$. Such a
function is always locally available by considering a least
squares cost with respect to local coordinates on the homogeneous
output space $M$.
In practice, local coordinates generally yield a
poor choice and the best Morse-Bott function candidates are
obtained by studying the global geometric structure of the
homogeneous space case by case. Since such homogeneous spaces are
generated as symmetry spaces of Lie groups there is often
significant structure that can be exploited in the construction of
global Morse-Bott candidate functions.  Once such a function
$\hat{f}$ is found it can be used to generate an invariant cost
function by setting
\[
f(y_1, y_2) = f(h(X,y_0),h(\hat{X},y_0)) = \hat{f}(h(\hat{X}^{-1}X,y_0)).
\]
for any $X$ and $\hat{X}$ such that $y_1 = h(X, y_0)$ and $y_2 =
h(\hat{X}, y_0)$.  It is straightforward to verify that $f$ is
well defined and invariant.  Moreover, if $\hat{f}$ has an unique
global minimum at $y_0$ and no further local minima, the function
$f$ will satisfy the conditions of Corollary \ref{cor:conv:hom}.
The construction is analogous to the development given in
\cite{toappear_Lageman.TAC}.

%%%%%%%%%%%%%%%%%%%%%%%%%%%%%%%%%%%%%%%%%%%%%%%%%%%%%%%%%%%%%%%%%%%%%%%%%%%%%%%%%%%%%%%
\subsection{Observers for the full system}\label{sec:5_design_full}

Section \ref{sec:5_design_induced} provided a constructive process
for the design of observers for projected systems.  Since the
projected system has a fully observed state, the gradient based
innovation is one of the most natural approaches to observer
design.  Since the projected system is a minimal observable
realisation of the \eqref{eq:sys_state} and \eqref{eq:sys_out} on
the Lie group, it is not to be expected that an observer
constructed separately on the Lie group could yield more
information about the state than that obtained from the projected
system.  Thus, the most `natural' observer on the Lie group is the
`lift' of the observer on the projected system onto the Lie group.
However, only the innovation is equivariant for the project system
observer and the main system cannot be unique lifted.  The
proposed solution is to lift the innovation and use the natural
synchronous internal model on the Lie group to define an observer.

For any reductive homogeneous space there is an invariant
horizontal distribution $H(X)$ and an invariant Riemannian metric
on $G$ such that the projection of the metric from $H(X)$ onto $T
M$ induces the invariant Riemannian metric on $M$.  Details of the
construction for the case of left invariant actions is provided in
Cheeger \etal \cite{ce}.  The case of right group actions and
right invariant metrics is entirely analogous.  The horizontal
distributions arising in this construction have the form discussed
in the proof of Theorem \ref{thm:inv:inno}, but the subspace $\h$
has to be chosen to satisfy some additional constraints that ensure the
metric construction is well defined.  Given the horizontal
distribution $H(X)$ there is a unique linear map, termed the lift,
\[
(\cdot)^{H(X)} : T_y M \rightarrow H(X)
\]
such that $T_X h_{y_0}((v)^{H(X)}) = v$.  That is $(v)^{H(X)} \in
H(X)$ denotes the lift of $v \in T_y M$.  Where the point $X \in
G$ is clear from context, or any choice is equivalent, we will
write $(v)^H = (v)^{H(X)}$.

The proposed observer for the system \eqref{eq:sys_state} and
\eqref{eq:sys_out} is given by
\begin{equation}\label{eq:obs}
\dot{\hat{X}} = \hat{X} u - \left( \grad_1 f(h_{y_0}(\hat{X}),y)\right)^H
\end{equation}
on the group.  The following proposition follows directly from its
construction.

\begin{proposition}\label{prop:proj_obs}
The observer \eqref{eq:obs} projects to the observer
\eqref{eq:obs:hom} on the homogeneous space.
\end{proposition}

Only the innovation term is a horizontal lift in \eqref{eq:obs}.
The full system is not a horizontal lift since $\hat{X}u \not\in
H(\hat{X})$ need not lie in the horizontal distribution and consequently
\[
\hat{X}u \not= \left( T_{\hat{X}} h_{y_0}( \hat{X} u ) \right)^H.
\]
The proposed observer, however, can be thought of as a more
general lift of projected system onto $G$ as shown in Proposition
\ref{prop:proj_obs}.  The convergence properties of the observer
\eqref{eq:obs} follow directly from the results on the homogeneous
space.
 
\begin{corollary}\label{cor:conv:grp}
Consider the left invariant system \eqref{eq:sys_state} with
complementary outputs \eqref{eq:sys_out} on a reductive
homogeneous output space. Let $f$ be and invariant cost function
such that $\hat{y}\mapsto f(\hat{y},y_0)$ is a Morse-Bott function
with unique global minimum $y_0$ and no further local minima.  Let
$E_r : G \times G \rightarrow G$, $E_r(\hat{X}, X) =
\hat{X}X^{-1}$ be the right invariant error on $G$
\cite{toappear_Lageman.TAC}.

Then for generic initial conditions $\hat{X}(0)$ and arbitrary,
admissible inputs $u$ the error $E_r$ converges asymptotically to $\stab(y_0)$.
\end{corollary}

\begin{proof}
This result is a straightforward consequence of Corollary
\ref{cor:conv:hom} and the fact that the canonical error on $M$ is
the projection of $E_r$.
\end{proof}

The observer \eqref{eq:obs} can also be derived directly on the
Lie group.  For this purpose we define the lift $\tilde{f}$ of the
cost function $f$ as the map $G \times G \rightarrow \R$,
\begin{equation}\label{eq:f_tilde}
\tilde{f}(\hat{X},X)=f(h_{y_0}(\hat{X}),h_{y_0}(X)).
\end{equation}

\begin{proposition}\label{prop:lift:cost}
Consider a right group action on a reductive homogeneous space.
Let $f : M \times M \rightarrow \R$ be an invariant cost function
on $M$ and let $\tilde{f}$ be given by \eqref{eq:f_tilde}.  The
gradient of the cost function $f$ and the lifted function
$\tilde{f}$ are related by
\[
\grad_1 \tilde{f}(\hat{X},X) = \left(\grad_1 f(h_{y_0}(\hat{X}),h_{y_0}(X))\right)^H,
\]
If $f$ is invariant under the action of $G$ on $M$, then the
lifted function $\tilde{f}$ is a right invariant cost function on
$G$.
\end{proposition}

\begin{proof}
For the first statement, note that for all $v\in T_{\hat{X}} G$
\[
\dd_1 \tilde{f}(\hat{X},X)(v) = \dd_1 f(\hat{X}y_0,\hat{X}y_0)
(T_{\hat{X}}h_{y_0} v).
\]
Hence for all $v\in T_{\hat{X}} G$
\[
\langle \grad_1 \tilde{f}(\hat{X},X),v\rangle = \langle \grad_1 f(\hat{X}y_0,\hat{X}y_0) ,T_{\hat{X}}h_{y_0} v \rangle.
\]
In particular,
\[
\langle \grad_1 \tilde{f}(\hat{X},X),v\rangle = 0
\]
for $v\in \ker T_{\hat{X}}h_{y_0}$.
Hence, $\grad_1 \tilde{f}(\hat{X},X)\in H(\hat{X})$.
On the other hand we have by our conditions on the Riemannian metrics that
for all $v \in H(\hat{X})$
\[
\langle \grad_1 \tilde{f}(\hat{X},X),v\rangle
=
\langle T_{\hat{X}}h_{y_0} \grad_1 \tilde{f}(\hat{X},X),T_{\hat{X}}h_{y_0} v\rangle .
\]
Hence $\grad_1 \tilde{f}$ is the horizontal lift of $\grad_1 f$.
If $f$ is invariant, then for all $X,\hat{X},Z\in G$ we have
\begin{multline*}
\tilde{f}(\hat{X}Z,XZ)=f(h_{y_0}(\hat{X}Z),h_{y_0}(XZ))
=f(h(Z,h_{y_0}(\hat{X})),h(Z,h_{y_0}(X))) \\
=f(h_{y_0}(\hat{X}),h_{y_0}(X))=\tilde{f}(\hat{X},X),
\end{multline*}
i.e. the lift is invariant too.
\end{proof}

The lifted cost function is then a candidate for the observer
construction for systems on a Lie group proposed in earlier work
\cite{toappear_Lageman.TAC}. This construction yields
\begin{equation}\label{eq:lift:obs}
\dot{\hat{X}} = \hat{X} u - \grad_1 \tilde{f}(\hat{X},X).
\end{equation}
This construction is made based on the understanding that the full
state $X$ is measured, however, in this specific case the state
$X$ is only used in the innovation term in a manner that can be
derived from the information contained in the measurement $y =
h(X, y_0)$.  The easiest way to see this is to note that by
Proposition \ref{prop:lift:cost} the observer \eqref{eq:lift:obs}
is the same observer as \eqref{eq:obs}.  As a consequence we
get the equivalence of the observer error trajectories.

\begin{corollary}
Consider the left invariant system \eqref{eq:sys_state} with
complementary outputs \eqref{eq:sys_out} on a reductive
homogeneous output space $M$. Let $f$ be and invariant cost
function such that $\hat{y}\mapsto f(\hat{y},y_0)$ is a Morse-Bott
function with unique global minimum $y_0$ and no further local
minima.  Let $E_r = \hat{X}X^{-1}$ be the right invariant error on
$G$ and $e = h(E_r,y_0)$ be the induced canonical error on $M$.
Then if $h(\hat{X}(0),y_0) = \hat{y}(0)$, the error trajectory
$E_r(t)$ generated by \eqref{eq:obs} with initial conditions
$\hat{X}(0)$ on $G$ projects to the error trajectory $e(t)$
generated by \ref{eq:obs:hom} with initial conditions
$\hat{y}(0)$.  That is
\[
h(E_r(t), y_0) = e(t).
\]
\end{corollary}

%%%%%%%%%%%%%%%%%%%%%%%%%%%%%%%%%%%%%%%%%%%%%%%%%%%%%%%%%%%%%%%%%%%%%%%
\section{Example: Attitude estimation}\label{sec:6_example}

In this section, we complete the analysis of the simple example
introduced in Examples~\ref{ex:attitude_1} and \ref{ex:attitude_2}
to demonstrate how the developments in the paper can be applied in
practice.  We consider the complementary output case discussed in
Example~\ref{ex:attitude_1}, that of estimating the attitude of a
rigid body based on measurements of a single vectorial direction
measured in the body-fixed-frame along with full measurement of
angular velocity.  There is an extensive literature concerning the
estimation of attitude of rigid-bodies based on body-fixed frame
measurements
\cite{Bonnabel2006_acc,2008_Bonnabel.TAC,2008_Bonnabel.ifac,Bonnabel2005_ch,2007_Crassidis.JGCD,2006_HamMah_ICRA,2008_Lageman_mtns,toappear_Lageman.TAC,2005_MahHamPfl-C64,2008_MahHamPfl.TAC,Maithripala2004,2004_Maithripala_acc,2008_Martin.ifac,2007_Martin_cdc,Metni_etal2006_CEP,Metni2005a,Sal1991_TAC,2007_Tayebi_cdc,TayMcG_2006_CST,ThiSan2003_TAC,2008_Vasconcelos.ifac,VikFos_2001_CDC}.
We mention in particular, the recent papers
\cite{2006_HamMah_ICRA,2008_MahHamPfl.TAC} that introduced the
``explicit complementary filter'' and the earlier work
\cite{Metni2005a,Metni_etal2006_CEP} that developed observers for
an inertial direction.   Both these observers are obtained as
specialisations of the general techniques proposed in this paper
for a natural choice of cost function.  We mention also the
independent work \cite{2008_Bonnabel.TAC,2008_Bonnabel.ifac} that
obtained the same observer as \cite{2008_MahHamPfl.TAC}.

The system is posed on $SO(3)$ (see Example~\ref{ex:attitude_1})
\begin{subequations}\label{eq:ex2}
\begin{align}
\dot{R} = & R\Omega_\times \label{eq:ex2_dotR}\\
y = & R^\top y_0 \label{eq:ex2_y}
\end{align}
\end{subequations}
with $y_0\in \frameA$ the `reference' direction in the inertial
frame.  The complementary output $y \in \frameB$ is the
observed orientation of the vectorial direction measured in the
body-fixed-frame.  The measurement equation \eqref{eq:ex2_y}
models the coordinate transformation induced by the relative
orientation of $\frameB$ with respect to $\frameA$.

The homogeneous output space is the unit sphere $S^2$.  By
Theorem~\ref{thm:lr}, the system \eqref{eq:ex2} projects to a
system on $S^2$.
Recalling the derivation in Example~\ref{ex:attitude_2} the
projected system can be written
\begin{equation}\label{eq:ex2:sp}
\dot{y} = -\Omega_\times y, \quad y(0) = h(X(0), y_0)
\end{equation}
on $S^2$.  We continue by designing an observer for system
\eqref{eq:ex2:sp} as discussed in Section \ref{sec:5_design}.

Consider the cost function $f(\hat{y},y)=(k/2)\|\hat{y}-y\|^2$ for
$k > 0$ a positive constant. It is straightforward to verify that
$f$ satisfies the conditions of Corollary \ref{cor:conv:hom}. That
is, $f\colon S^2\times S^2\rightarrow\R$ is invariant under the
right action of $SO(3)$ on $S^2\times S^2$, and $f$ is Morse-Bott
in the first argument with a unique global minima at $\hat{y} =
y$.  The constant $k$ is introduced to provide a tunable gain in
the observer design: adjusting $k$ scales the eigenvalues of the
Hessian of $f$ and tunes the exponential rate of the convergence
of the error dynamics.  Note that
\[
f(\hat{y},y)=\frac{k}{2}\|\hat{y}-y\|^2 = k(1 - \langle \hat{y}, y\rangle).
\]
The term  $(1 - \langle \hat{y}, y\rangle) = 1 - \cos (\angle
(\hat{y}, y)) $ is locally quadratic in the angle $\angle(\hat{y},
y))$.  The references
\cite{Metni2005a,Metni_etal2006_CEP,2006_HamMah_ICRA,2008_MahHamPfl.TAC}
used $(1 - \langle \hat{y}, y\rangle)$ directly as a Lyapunov
function.  This cost is also closely related to the quaternion
cost used in earlier works \cite{VikFos_2001_CDC,ThiSan2003_TAC}
as discussed in the appendix of \cite{2008_MahHamPfl.TAC}.

The Riemannian metric on $S^2$ induced by the direct embedding of
$S^2$ in $\R^3$ is invariant under the right action $(R, x) \mapsto
R^\top x$ of $SO(3)$ on $\R^3$.  In particular, $(R^\top v)^\top
R^\top w = v^\top R R^\top w = v^\top w$ for $v, w \in T_y S^2 \subset \R^3$.
The gradient of $f$ is computed from the relationships $\langle
\grad_1 f(\hat{y},y), w \rangle = \dd_1 f(\hat{y},y)[w]$ and
$\grad_1 f(\hat{y},y) \in T_{\hat{y}} S^2$. 
Here $\dd_1 f(\hat{y},y)$ denotes the differential of $f$ with respect to the first variable, i.e. $\hat{y}$.
One has
\[
\dd_1 f(\hat{y},y)[w] = w^\top (\hat{y} - y) = k w^\top \underbrace{(I - \hat{y} \hat{y}^\top )(\hat{y} - y)}_{\in T_{\hat{y}} S^2}
\]
since $w^\top  = w^\top (I - \hat{y} \hat{y}^\top ) $ as $w \in
T_{\hat{y}} S^2$.  It follows that
\[
\grad_1 f(\hat{y},y) = - k \left( I - \hat{y}\hat{y}^\top \right) y
\]
since $\hat{y}$ is in the kernel of the projection
$(I-\hat{y}\hat{y}^\top)$.

The proposed observer  for the projected output system is \eqref{eq:obs:hom}
\begin{equation}\label{eq:obs:s2}
\dot{\hat{y}} = -\Omega_\times \hat{y} + k (I-\hat{y}\hat{y}^\top)y.
\end{equation}
Note that this observer equation is fully posed as a system on
$S^2$ and does not depend on the underlying system on $SO(3)$.
Since $\hat{y}\mapsto f(\hat{y},y)$ is a Morse function with only
one global minimum and no further local minima, Corollary
\ref{cor:conv:hom} applies and the canonical error on $S^2$ for
this observer converges for generic initial conditions to $y_0$.

It is of interest to provide an equivalent form for
Eq.~\ref{eq:obs:s2}.  Observe that
\begin{equation}\label{eq:alg_1}
k(y \hat{y}^\top - \hat{y}y^T) \hat{y} = k(y -  \hat{y}y^T\hat{y}) = k\left( I - \hat{y}\hat{y}^\top \right)  y = - \grad_1 f(\hat{y}, y)
\end{equation}
Moreover, it is easily verified that
\begin{equation}\label{eq:alg_2}
(\hat{y} y^\top - y\hat{y}^\top) = ( \hat{y} \times y )_\times
\end{equation}
Thus, \eqref{eq:obs:s2} can be written
\begin{equation}\label{eq:dot_hat_y}
\dot{\hat{y}} = -\Omega_\times \hat{y} + k( \hat{y} \times y ) \times \hat{y}.
\end{equation}
This observer was originally studied by Metni \etal
\cite{Metni2005a,Metni_etal2006_CEP}.  The derivation given in
earlier work is based on a Lyapunov analysis and includes an
integral term for compensation of gyro bias that does not fit the
analysis framework presented in this paper, however, the main
proportional terms in \cite{Metni2005a,Metni_etal2006_CEP} are
exactly \eqref{eq:dot_hat_y}.

To lift \eqref{eq:obs:s2} to $SO(3)$ it is necessary to construct
a right invariant horizontal distribution for which the right
invariant metric on $SO(3)$ projects to the metric $\langle u,
v\rangle$ on $S^2$ via the differential of the group action.
Define
\[
\h = \{ \omega_\times \in \g \;\vline\; \tr(\omega_\times^\top (y_0)_\times) = 0 \} \equiv \{ \omega_\times \;\vline\; \omega \in \R^3, \omega^\top y_0= 0 \}
\]
The subspace $\h$ is chosen as the orthogonal complement of the
subalgebra of the stabiliser of $h$ under the inner product
$\langle \Omega_\times, \Psi_\times \rangle = \tr(
\Omega_\times^\top \Psi_\times)$ defined on $\g$.  Physically,
this set corresponds to angular velocities expressed in the
inertial frame $\frameA$ that are orthogonal to the reference
direction $y_0$.  The horizontal distribution $H(\hat{R})$ is
defined in the standard manner\footnote{See also the proof of
Theorem~\ref{thm:inv:inno}.}
\[
H(\hat{R}) = \{ \omega_\times \hat{R} \;\vline\; \omega_\times \in \h \}.
\]
For any $\omega_\times \in \h$ there is an element $\Omega_\times
= \Ad_{\hat{R}^\top} \omega_\times$, where $\Ad_{\hat{R}^\top}
\omega_\times = \hat{R}^\top \omega_\times \hat{R}$ is the adjoint
operator, associated with the angular velocity $\omega$ expressed
in the body-fixed-frame.  Thus, one can also write
\[
H(\hat{R}) = \{ \hat{R} \Omega_\times \;\vline\; \Omega_\times = \Ad_{\hat{R}^\top} \omega_\times, \omega_\times \in \h\}.
\]
Next we consider how to compute the lift $(v)^H$ of a tangent
vector in $T_{\hat{y}} S^2$ into $H(\hat{R})$. 
Let $v \in T_{\hat{y}} S^2$ and $\hat{R} \in SO(3)$ such that
$h(\hat{R}, y_0) = \hat{y}$.  Define $\bar{\Omega}(v)$ by the solution to
\[
\bar{\Omega}(v) \times \hat{y} = v, \quad \bar{\Omega}(v)^\top \hat{y} = 0.
\]
Since $v \in T_{\hat{y}} S^2$ is orthogonal to $\hat{y}$ then it
is clear that this vector exists and is unique.  Physically,
$\bar{\Omega}(v)$ is the instantaneous body-fixed-frame velocity
that generates the observed velocity $v$ for the projected
dynamics of $\hat{y}$, and has zero yaw component around $\hat{y}$.
The horizontal lift $(v)^H$ is defined by
\[
(v)^H := \left( \Ad_{\hat{R}} {\bar{\Omega}(v)}_\times \right) \hat{R} = \hat{R}{\bar{\Omega}(v)}_\times.
\]
The element ${\omega(v)}_\times = \Ad_{\hat{R}} {\bar{\Omega}(v)}_\times \in
\h$ corresponds to the element in the subspace $\h$ representing
the lifted angular velocity expressed in the inertial frame.

It remains to show that the metric $\langle u, v\rangle = u^\top
v$ on $S^2$ is induced by a right invariant metric on $H(\hat{R})$.
Observe that for $v, w \in T_{\hat{y}} S^2$ then for any $\hat{R}$
such that $h(\hat{R}, y_0 ) = \hat{y}$, one has
\begin{align}
v^\top w = &  \hat{y}^\top {\bar{\Omega}(v)}_\times^\top {\bar{\Omega}(w)}_\times  \hat{y}
 =   {\bar{\Omega}(v)}^\top \hat{y}_\times^\top \hat{y}_\times \bar{\Omega} (w)  \notag\\
 = &  {\bar{\Omega}(v)}^\top (I - \hat{y} \hat{y}^\top) \bar{\Omega} (w)
=   {\bar{\Omega}(v)}^\top \bar{\Omega} (w)  \notag\\
= & \frac{1}{2} \tr\left( {\bar{\Omega}(v)}_\times^\top {\bar{\Omega}(w)}_\times\right)
=  \frac{1}{2} \langle  {\bar{\Omega}(v)}_\times \hat{R},  {\bar{\Omega}(w)}_\times \hat{R} \rangle \label{eq:metric_G}
\end{align}
where the final line is the natural right invariant metric on the
Lie group $SO(3)$ scaled by $1/2$.  The scaling factor is
associated to the structure of the matrix representation used.
The
above argument verifies that the projection $T_{\hat{R}}
h_{y_0}$ projects the right invariant metric \eqref{eq:metric_G}
restricted to $H(\hat{R})$ to the embedded Euclidean metric on
$T_{\hat{y}} S^2$.

From Equation~\ref{eq:obs} the lift of the observer
\eqref{eq:lift:obs} is
\begin{equation}\label{eq:obs_SO(3)}
\dot{\hat{R}}= \hat{R} \Omega_\times + k \left( (I-\hat{y}\hat{y}^\top)y \right)^{H}
\end{equation}
From Proposition~\ref{prop:proj_obs} it follows that this observer
projects to the observer on $S^2$, Eq.~\ref{eq:obs:s2}.
Corollary~\ref{cor:conv:grp} states that $\hat{R}$ converges to
the set $\{\hat{R} \in SO(3) \;\vline\; h(\hat{R},y_0) = y \}$,
that is, $\hat{R}$ is identified up to the unobservable rotation
around the measured direction.

It is of interest to provide an explicit form for
Eq.~\ref{eq:obs_SO(3)}.  Recalling \eqref{eq:alg_1} and \eqref{eq:alg_2} one has
\[
\grad_1 f(\hat{y}, y) = k ( \hat{y} \times y )_\times \hat{y}
\]
It is easily verified that ${\bar{\Omega}(\grad_1 f)}_\times = k (
\hat{y} \times y)_\times$.   Thus, the lifted observer can  be
written
\begin{align}
\dot{\hat{R}} = &  \hat{R} u - k \left( \Ad_{\hat{R}} ( \hat{y} \times y )_\times  \right) \hat{R}  \notag \\
= &  \hat{R} \left(u - k (\hat{y} \times y )_\times \right) = \hat{R} \left(u + k (y \times \hat{y})_\times \right)
\label{eq:obs_lift}
\end{align}
Equation \ref{eq:obs_lift} is the explicit complementary filter
proposed in \cite[Eq.~(32)]{2008_MahHamPfl.TAC} excluding the
integral introduced in that paper to compensate gyro bias.

The observer \eqref{eq:obs_SO(3)} can also be obtained as a direct
gradient of a lifted cost function.  The lifted cost function
$\tilde{f}$ of $f$ is given by
\[
\tilde{f}(\hat{R},R) = \frac{k}{2}\|\hat{R}^\top y_0 - R^\top y_0\|^2 = \frac{k}{2}\tr\left( (\hat{R}-R)(\hat{R}-R)^\top y_0 y_0^\top\right).
\]
The differential of the lifted cost function with respect to the first coordinate is
\[
\dd_1 \tilde{f}(\hat{R},R)[\hat{R}\Omega_\times] =  k\tr\left( \hat{R}\Omega_\times (\hat{R}-R)^\top y_0y_0^\top \right) - k\tr\left( (\hat{R}-R)\Omega_\times \hat{R}^\top y_0y_0^\top\right)
\]
Rearranging terms and using the notation
$\pp_{\so(3)}(A)=(A-A^\top)/2$ for matrix projection onto the
subspace of anti-symmetric matrices, the differential can be rewritten
\begin{align*}
\dd_1 \tilde{f}(\hat{R},R)(\hat{R}\Omega_\times) = & k\tr\left( \Omega_\times \pp_{\so(3)} \left( (\hat{R}-R)^\top y_0y_0^\top \hat{R} \right) \right) \\
= & 2k \tr\left( \Omega_\times \pp_{\so(3)} \left( (\hat{y}-y)\hat{y}^\top  \right) \right) \\
= & 2k \tr\left( \Omega_\times  (\hat{y}y^\top -y \hat{y}^\top ) \right).
\end{align*}
Using the metric \eqref{eq:metric_G} one obtains the gradient
\[
\grad_1 \tilde{f}(\hat{R},R) = - k\hat{R} (\hat{y}y^\top -y \hat{y}^\top )= - k\hat{R} (\hat{y} \times y ).
\]
Note that factor of $1/2$ in the metric cancels the factor of $2$
in the differential.  The gradient-like observer on the Lie-group
is
\[
\hat{R} = \hat{R} u - k\hat{R} (\hat{y} \times y ) = \hat{R} u - (\grad_1 f(\hat{R}, R))^H
\]
(recalling \eqref{eq:dot_hat_y}) and one obtains the observer
\eqref{eq:obs_lift} as expected.

%%%%%%%%%%%%%%%%%%%%%%%%%%%%%%%%%%%%%%%%%%%%%%%%%%%%%%%%%%%%
\section{Conclusions}\label{sec:7_conclusions}

This paper provides a comprehensive analysis of the design of
observers for invariant systems posed on finite-dimensional
connected Lie groups with measurements generated by a
complementary group action on an associated homogeneous space. The
observer synthesis problem can be tackled for the projected system
on the homogeneous output space based on a canonical decomposition
of the observer dynamics into a synchronous internal model and an
equivariant innovation.  A gradient-like construction for the
innovation term was proposed that leads to strong stability
properties of the observer, both for the projected system on the
homogeneous output space, and for the lifted system on the
Lie-group (stable to the unobservable subgroup).

%%%%%%%%%%%%%%%%%%%%%%%%%%%%%%%%%%%%%%%%%%%%%%%%%%%%%%%%%%%%
\bibliographystyle{abbrv}
\bibliography{hfilter}
%%%%%%%%%%%%%%%%%%%%%%%%%%%%%%%%%%%%%%%%%%%%%%%%%%%%%%%%%%%%
\end{document}